\documentclass[letterpaper,twoside]{article}
\usepackage{fancyhdr}

\usepackage[left=3.3cm, top=3.2cm, right=3.3cm, bottom=3.2cm, marginparwidth=2.6cm, marginparsep=5mm]{geometry}
\usepackage{tikz-cd}
\usepackage{amsmath} 
\usepackage{enumerate}
\usepackage{amsfonts}
\usepackage{amssymb}

\usepackage{mathrsfs}
\usepackage{faktor}
\usepackage{xcolor}
\usepackage{array}
\usepackage{blkarray}

\usepackage[disable]{todonotes}

\usepackage{xparse}
\usepackage{amsthm}
\usepackage{indentfirst}
\usepackage{float}
\usepackage{booktabs}

\usepackage{titlesec} 
\usepackage{ stmaryrd }
\usepackage{charter}
\usepackage[bitstream-charter]{mathdesign}
\usepackage{tgheros}
\usepackage{bbm}

\usepackage[OT2,T1]{fontenc}
\SetMathAlphabet{\mathcal}{normal}{OMS}{cmsy}{m}{n}

\titleformat{\section}[hang]{
	\usefont{T1}{qhv}{b}{n}\selectfont} 
{} 
{0em}
{\hspace{-0.4pt}\Large \thesection\hspace{0.6em}}
\titleformat{name=\section,numberless}[display]
{
	\usefont{T1}{qhv}{b}{n}\selectfont} 
{} 
{0em}
{\Large}

\renewcommand{\refname}{References}



\makeatletter \g@addto@macro\@floatboxreset\centering

\makeatother

\renewcommand{\subsectionmark}[1]
{}

\newcommand{\helv}{
	\fontfamily{phv}\fontseries{b}\fontsize{9}{11}\selectfont}
\edef\restoreparindent{\parindent=\the\parindent\relax}
\usepackage{parskip}
\restoreparindent

\usepackage[ocgcolorlinks]{hyperref} 

\definecolor{linkcolour}{rgb}{0,0.2,0.6}
\hypersetup{colorlinks,breaklinks,urlcolor=linkcolour, linkcolor=linkcolour}
\DeclareDocumentCommand{\newfaktor}{s m O{0.5} m O{-0.5}}{
	\setbox0=\hbox{\ensuremath{#2}}
	\setbox1=\hbox{\ensuremath{\diagup}}
	\setbox2=\hbox{\ensuremath{#4}}
	\raisebox{#3\ht1}{\usebox0}
	\mkern-5mu\ifthenelse{\equal{#1}{\BooleanTrue}}
	{\diagup}
	{\rotatebox{-44}{\rule[#5\ht2]{0.4pt}{-#5\ht2+#3\ht0+\ht0}}}
	\mkern-4mu
	\raisebox{#5\ht2}{\usebox2}
}

\usepackage{xpatch}
\makeatletter
\xapptocmd{\@sect}{\csname #1mark\endcsname{#7}}{}{}
\makeatother

\newtheoremstyle{mystyle}
{.2\baselineskip}
{.2\baselineskip}
{\itshape}
{}
{\fontsize{10.5pt}{0pt}\scshape }
{}
{.5em }
{}

\theoremstyle{mystyle}
\newtheorem{lemma}{Lemma}[section]

\newtheorem{theorem}[lemma]{Theorem}

\newtheorem{proposition}[lemma]{Proposition}
\newtheorem{prop}[lemma]{Proposition}

\newtheoremstyle{roman}
{.4\baselineskip}
{.2\baselineskip}
{\normalfont \setlength\parskip{.2\baselineskip}}
{}
{\scshape}
{}
{.5em }
{}
\theoremstyle{roman}

\newtheorem{notat}[lemma]{Notation}

\newtheorem{remark}[lemma]{Remark}

\newtheorem*{thm*}{Theorem}

\makeatletter
\renewenvironment{proof}[1][\proofname]{\par
	\pushQED{\qed}
	\normalfont \topsep0\p@\@plus0\p@\relax
	\trivlist
	\item\relax
	{\itshape
		#1\@addpunct{.}}\hspace\labelsep\ignorespaces
}{
	\popQED\endtrivlist\@endpefalse
}
\makeatother

\newcommand{\Rep}{\operatorname{Rep}}

\newcounter{saveenumerate}
\makeatletter
\newcommand{\enumeratext}[1]{
	\setcounter{saveenumerate}{\value{enum\romannumeral\the\@enumdepth}}
\end{enumerate}
#1
\begin{enumerate}[i)]
	\setcounter{enum\romannumeral\the\@enumdepth}{\value{saveenumerate}}
}
\makeatother

\include{functions}

\newcommand{\m}{\mathfrak{m}}

\newcommand{\GSp}{\operatorname{GSp}}

\newcommand{\MM}{\mathcal{M}}
\renewcommand{\M}{\operatorname{M}}

\renewcommand{\opp}{\mathrm{op}}

\renewcommand{\H}{\mathcal{H}}

\newcommand{\V}{\mathcal{V}}
\newcommand{\Zar}{\textrm{Zar}}

\newcommand{\calX}{\mathcal{ X}}
\newcommand{\calA}{\mathcal{ A}}
\newcommand{\calY}{\mathcal{Y}}

\newcommand{\calH}{\mathcal H}

\newcommand{\calM}{\mathcal{M}}

\newcommand{\Aff}{\operatorname{Aff}}
\newcommand{\Jac}{\operatorname{Jac}}
\renewcommand{\M}{\mathcal{M}}
\newcommand{\X}{\mathcal{X}}
\newcommand{\XX}{\mathcal{X}}
\renewcommand{\Aut}{\operatorname{Aut}}

\newcommand{\crys}{\mathrm{crys}}
\newcommand{\LGr}{\operatorname{LGr}}

\newcommand{\YY}{\mathcal{Y}}
\newcommand{\univ}{\mathrm{univ}}
\renewcommand{\cris}{\crys}

\title{Bounds on the number of rational points of curves in families}
\author{Pedro Lemos and Alex Torzewski}
\begin{document}
\maketitle
\begin{abstract}
In this note, we give an alternative proof of uniform boundedness of the number of integral points of smooth projective curves over a fixed number field with good reduction outside of a fixed set of primes. We use that due to Bertin--Romagny, the Kodaira--Parshin families constructed by Lawrence--Venkatesh can themselves be assembled into a family. We then repeat Lawrence--Venkatesh's study of the $p$-adic period map, together with the comparison of nearby fibres. 
\end{abstract}
\section{Introduction}
Suppose we are given a smooth variety $\mathcal{X}/\OO_{K,S}$ for some number field $K$ and finite set of primes $S$. In \cite{LV}, Lawrence--Venkatesh suggest the following method for studying its integral points. Suppose that $\calX$ is additionally equipped with a smooth projective family $\calY\overset{\pi}{\to} \calX$. To points of $\calX$ we may associate the Galois representation given by the étale cohomology of the fibre $\calY_x$. This can be done globally and locally:
\[ \begin{tikzcd}
\calX({\OO}_{K,S})\ar{d}[swap]{x \mapsto H^1_\et(\calY_{x,\bar K},\q_p)} \ar[hookrightarrow]{r} & \calX(\OO_{K_v})\ar{d}{x\mapsto H^1_\et(\calY_{x,\overline{K_v}},\q_p)}\\
\Rep_{\q_p}(G_K)/\sim  \ar{r}[swap]{\textrm{restr.}}& \Rep_{\q_p}(G_{K_v})/\sim
\end{tikzcd}  \]
Here $G_F$ denotes the absolute Galois groups of a field $F$ and $\Rep_{\q_p}(G_{F})/\sim $ denotes isomorphism classes of $\q_p$-valued $G_F$-representations. The Faltings--Serre method shows that the image of the left hand map is finite if the $H^1_\et(\calY_{x,\bar K},\q_p)$ are known to be semisimple \cite[V.2]{FaltingsWustholz}. This semisimplicity is known to hold in the case when the fibres are abelian varieties \cite[Satz 3]{FaltingsSemisimplicity}, or disjoint unions thereof, and is expected in general. For $\mathcal{X}$ a smooth projective curve of genus $g\ge 2$, Lawrence--Venkatesh \cite{LV} constructed a ``Kodaira--Parshin'' family, whose fibres are disjoint unions of abelian varieties and for which they could show that for suitable choices of $p$ and $v\mid p$, the right hand map is finite-to-one. As a result, they were able to recover the Mordell Conjecture on the finiteness of $\calX(\OO_{K,S})$.

Lawrence--Venkatesh study $\calX(\OO_{K_v})\to \Rep_{\q_p}(G_{K_v})/\sim$ via $p$-adic period mappings. For $p$ coprime to $S$ and $v\mid p$ unramified, the Galois representations $H^1_\et(\calY_{x,\overline{K_v}},\q_p)$ are crystalline and their associated filtered $F$-isocrystal is given by $(H^1_\cris(\calY_{\overline x}/K_v),\phi, F^\bullet)$, where $\phi$ is the action of absolute Frobenius and $F^\bullet$ is the Hodge filtration on de Rham cohomology. Within a residue disk $\Omega$, the pair $(H^1_\cris(\calY_{\overline x}/K_v),\phi) $ is constant, so the variation of the filtered $F$-isocrystal can be recorded simply by the period map $\Phi \colon \Omega \to \calH$ measuring the variation of its Hodge filtration within some suitable period domain. The choices of Hodge filtration on $H^1_\cris(\calY_{\overline x}/K_v)$ which define isomorphic filtered $F$-isocrystals are contained in a closed algebraic subvariety $Z\subseteq \calH$. For the Kodaira--Parshin family, Lawrence--Venkatesh showed that the image of the period map is never contained within $Z$ and consequently that $\Phi^{-1}(Z)$ is a finite set. 

In this note, we show that existence of a ``relative Kodaira--Parshin family'' $\calY \to \calX\to \calM$ whose fibre over $u \in \calM$ recovers the Kodaira--Parshin family allows one to deduce uniform bounds on the size of $\Phi^{-1}(Z)\cap \calX_u(\OO_{K_v})$, i.e.\ the fibres of the map $\calX_u(\OO_{K_v})\to \Rep_{\q_p}(G_{K_v})/\sim$. Since such a family is known to exist by work of Bertin--Romagny \cite{BertinRomagny}, we are able to deduce:
\begin{theorem}\label{t:main}
	Let $F$ be a number field and $S$ be a finite set of primes. Let $g\geq 2$ be an integer. There exists a constant $N= N(F,S,g)$ such that
	\[ \#\mathcal{C}(F)\leq N\] for all smooth projective curves $\mathcal{C}/\OO_{F,S}$ of genus $g$ with good reduction outside of $S$.
\end{theorem}
This is traditionally seen by combining the Shafarevich Conjecture, i.e.\ finiteness of the number of isomorphism classes of smooth projective curves $\cal{C}/\OO_{K,S}$, and the Mordell Conjecture, both of which are known due to Faltings \cite[Kor.\ 1, Satz 3]{FaltingsSemisimplicity}.

The idea of the proof is that for $u,u' \in \mathcal{M}(\OO_{K_v})$ sufficiently close, in a way that can be made precise, the period maps are equal modulo $\mathfrak{m}_{K_v}^k$. For $k$ sufficiently large, a simple Newton polygon argument shows then shows that the number of zeroes of the divided powers power series vanishing on $\Phi^{-1}(Z)$ has a common bound. This method does not require any global properties of the family with fibres recovering the Kodaira--Parshin families beyond its existence and smoothness.

In \cite{LV}, they able to obtain finiteness without the assumption of semisimplicity of the $H^1_\et(\calY_{x,\bar{K}},\q_p)$, i.e. avoiding \cite[Satz 3]{FaltingsSemisimplicity}. This can be maintained in the proof of Theorem \ref{t:main}.

In Remark \ref{r:CHM}, we sketch an alternative proof of Theorem \ref{t:main} in the style of Caporaso--Harris--Mazur \cite{CHM}, using the stronger input of non-Zariski density within $\calX \ti_{\calM} U$ for arbitrary subvarieties $U\subseteq \M$. Both methods rely on the existence of a relative Kodaira--Parshin family and Lawrence--Venkatesh's study of its fibres.

\subsubsection*{Acknowledgements}
We are grateful to Alex Betts and Netan Dogra for helpful discussions and especially to the latter for drawing our attention to the work of Caporaso--Harris--Mazur. We also wish to particularly thank Marco Maculan for explaining how a relative Kodaira--Parshin family is a consequence of work of Bertin--Romagny and Matthieu Romagny for being available to explain the subtleties of their construction. Additionally, we are very grateful for the comments of an anonymous referee, which improved both its mathematical content and its readability.
	\section{$p$-adic period mappings}\label{s:background}
	We use this section to fix notation and recall some basic properties of period mappings.
	
	\begin{notat}
		\begin{itemize}
			\item Let $L/\q_p$ be a finite unramified field extension with ring of integers $\OO_L$.
			\item Write $\mathfrak{m}_L$ for the maximal ideal of $\OO_L$ and $k_L$ for the residue field.
			\item Fix an algebraic closure $\bar L/L$ and let $G_L$ denote the absolute Galois group of $L$.
			\item Fix a choice of smooth projective family $\pi \colon \YY\to \XX$ over $\OO_L$.
			\item For $x \in \XX(\OO_L)$, write $\YY_x$ for the fibre above $x$. Fix $i\ge 0$ and let $\rho_x \colon G_L \to \GL(H^i_\et(\YY_{x, \bar L},\q_p))$ be the corresponding $p$-adic Galois representation corresponding to $x$ (we shall take $i=1$ in subsequent sections). Our assumptions ensure that $\rho_x$ is crystalline for all $x$.
		\end{itemize}
	\end{notat}
	The variation of $\rho_x$ can be studied by means of period mappings. More specifically, one can study the variation of the filtered $F$-isocrystal $(H^i_\crys(\mathcal{Y}_{\overline{x}}/L), \phi, F^\bullet)$ associated to $\rho_x$ via $p$-adic Hodge theory. Here $\phi$ denotes the action of absolute Frobenius and $F^\bullet$ denotes the Hodge filtration on de Rham cohomology on $ H^i_\dR(\mathcal{Y}_x/L)\simeq H^i_\crys(\mathcal{Y}_{\bar x}/L)$.
	
	Fix a choice of residue disk $\Omega \subseteq \XX(\OO_L)$. Since for all $x\in \Omega$ the fibres $\mathcal{Y}_x$ have the common special fibre, the $F$-isocrystal underlying $(H^i_\crys(\mathcal{Y}_{\overline{x}}/L), \phi, F^\bullet)$ is constant within $\Omega$. As such, to record the variation of the filtered $F$-isocrystal, we need only record the data of the Hodge filtration within $H^i_\crys(\mathcal{Y}_{\bar x}/L)$. This yields a period map
	\begin{equation*} \Phi\colon \Omega \to \H(L), \end{equation*}
	where $\H$ is some suitable (completed) period domain. For the purposes of this overview, $\H$ can be taken to be the Grassmannian parametrising flags of the same type as the Hodge filtration on the fibres. Since $\YY_{x}$ is defined over $\OO_L$, the Hodge filtration is in fact defined over $H^i_\crys(\mathcal{Y}_{\bar x}/\OO_L)$ and since de Rham cohomology together with its Hodge filtration commutes with reduction modulo $\m_L^k$, the filtration is given by the (compatible) sequence of Hodge filtrations on mod $\m_L^k$ crystalline cohomology coming from the mod $\m_L^k$ crystalline--de Rham comparisons. As a result, we can instead consider $\Phi$ to take values in the smooth $\OO_L$-scheme $\H$ whose $\OO_L$-points parametrise $\OO_L$-flags within $H^i_\crys(\mathcal{Y}_{\bar x}/\OO_L)$ and its reduction modulo $\m_L^k$ coincides with the mod $\m_L^k$ period maps
	\begin{equation*} \overline{\Phi}  \colon \overline{\Omega} \longrightarrow \H (\OO_L/\m_L^k). \label{eq:redperiodmap}  \end{equation*}
	
	Alternatively, one can fix a base point $x_0\in \Omega$ and consider the variation of the Hodge filtration of the fibres when parallel transported back to $H^1_\dR(\mathcal{Y}_{x_0}/L)\simeq H^1_\crys(\calY_{\overline{x_0}}/L)$ using the Gauss--Manin connection and a local trivialisation of $R^if_* \Omega^\bullet_{\YY/\XX}$ respecting the Hodge filtration. These constructions agree as parallel transport via Gauss--Manin $H^1_\dR(\YY_{x}/L)\overset{\sim}{\to } H^1_\dR(\YY_{x_0}/L)$ precisely coincides with composing the crystalline--de Rham comparison isomorphisms for $x$ and $x_0$.

	Multiple choices of filtration on $(H^1_\crys(\YY_{\bar x}/L),\phi)$ will define isomorphic Galois representations. Explicitly, if $\alpha$ is any $L$-linear automorphism centralising $\phi$, then $F^\bullet$ and $\alpha(F^\bullet)$ will be isomorphic via $\alpha$. Alternatively, two points $x,x' \in \Omega$ for which $\rho_x$ and $\rho_{x'}$ are isomorphic have image under $\Phi$ lying within the same orbit under the centraliser $Z(\phi)$. As such, they also satisfy the weaker property of lying in the same orbit of $Z(\phi^{[L:\q_p]})$. These $Z(\phi^{[L:\q_p]})$-orbits have the advantage that they are closed algebraic subspaces of $\mathcal{H}(\OO_L)$ since $\phi^{[L: \q_p]}$ is $L$-linear.
	
	To obtain uniform bounds on the number of points within a fixed isomorphism class of Galois representations, we consider the preimage under $\Phi$ of closed algebraic subspaces. To do so requires further understanding of the structure of period maps. Via the description above, power series which locally define period maps satisfy the same properties as those defining flat sections of connections, and these are much more rigidly constrained than arbitrary locally convergent power series.
	
	Let $(\mathcal{V},\nabla)$ be any vector bundle with flat connection on $\XX$. Fix a base point $x_0 \in \XX(\OO_L)$ and a basis $\{\overline{v}_1,...,\overline{v}_n\}$ of $\mathcal{V}_{\overline{ x_0}}$. If $U\subseteq \XX$ is a sufficiently small affine open, we may lift the $\overline{v}_i$ to a basis $\{ v_1,...,v_n\}$ trivialising $\mathcal{V}$ over $U$. With respect to this choice of basis we may write
	\begin{equation*} \nabla(\underline{f}) = d\underline{f} + A\underline{f}   ,  \end{equation*}
		where $\underline{f}=(f_1,...,f_n)$ with $f_i\in \OO_U$ and $A$ is an $(n\ti n)$-matrix of sections of $\Omega^1_{\OO_U/\OO_L}$.
		
		Let $\Omega\subseteq \XX(\OO_L)$ be the residue disk containing $x_0$. Since $\mathcal{X}$ is smooth at the special fibre $\overline{x_0}$, its formal neighbourhood is isomorphic to $\Spf \OO_L[[t_1,...,t_r]]$, where $r=\dim \XX_L$ and, by definition, $\Omega$ consists of its the $\OO_L$-points with $x_0$ being the origin. Solutions to flat connections need not exist formally\footnote{That is, there need not be non-zero $\OO_L[[t_1,...,t_r]]$-valued elements of the kernel of $\nabla$.} in the integral case.
				
	\begin{notat}
		For any commutative ring $R$, let $R\la \la t_1,...,t_r\ra \ra$ denote the $n$-dimensional \emph{divided power algebra} over $R$. This is the free commutative $R$-algebra on symbols $t_i^{[s_i]}$, $i\in \{1,...,r\}$, $s_i\ge 0$, subject to the relations
		\[t_i^{[s_i]}\cdot t_i^{[s'_i]}=
		\left(\frac{(s_i+s'_i)!}{s_i! s'_i!}\right)t_i^{[s_i+s'_i]}\]
		(the fraction being considered as an integer). We consider $R\la \la t_1,...,t_r\ra \ra $ as a $R[[t_1,...,t_r]]$-algebra via $t_i^{s_i}\mapsto (s_i)!t_i^{[s_i]}$. When $R$ is an integral domain of characteristic zero, $R\la \la t_1,...,t_r\ra \ra$ can be considered as a subring of the ring of formal power series $\operatorname{Frac}(R)[[t_1,...,t_r]]$ via $t_i^{[s_i]}\mapsto \frac{t_i^{s_i}}{s_i!}$. 
	\end{notat}
Divided power algebras are minimal with respect to being able to run Picard--Lindel\"of's method of successive approximations to find solutions to flat connections:
	\begin{lemma} For any initial condition $\underline {f_{0}}\in \V_{0}$, there exists a unique $\underline{f}=(f_1,...,f_n)$ with $f_i \in \OO_L\la \la t_1,...,t_r\ra\ra$ for which $\nabla(\underline{f})=0$ and whose evaluation at $0$ is $\underline{f_{0}}$.
	\end{lemma}
\begin{proof}
	See, for example, \cite[Prop.\ 3.1.1.]{KatzTdD}\footnote{Katz's result is stated in the context of $F$-crystals, but the additional structure is not used for this result.}.
\end{proof}
In particular, period mappings are locally defined in terms of divided powers power series.

	Since 
	\begin{equation}v_p\left(\frac{1}{(p^r)!}\right)=- \left(\frac{p^r-1}{p-1} \right),\label{eq:valuationoffactorial}  \end{equation}
	the valuation of coefficients of $\sum_{\underline{s}\ge 0}\frac{a_{\underline{s}}}{|\underline{s}|!} t_1^{s_1}...t_r^{s_r} \in \OO_L\la \la t_1,...,t_r\ra \ra$ satisfies $v_p(a_{\underline{s}}/|\underline{s}|!) > -|\underline{s}|/(p-1)$ as soon as $|\underline{s}|>0$. As a result, elements of $\OO_L\la \la t_1,...,t_r\ra \ra$ converge on the open disk of radius $\left(\frac{1}{p}\right)^\frac{e(L/\q_p)}{p-1}$, where $e(L/\q_p)$ denotes the ramification degree of $L/\q_p$. Assuming that $L$ is unramified and $p\neq 2$, convergence holds on the entire residue disk $\mathfrak{m}_L\ti...\ti \mathfrak{m}_L$.
	
	
	We may define a reduction modulo $\m_L^k$ map $\OO_L\la\la t_1,...,t_r\ra \ra\to \OO_L/\m_L^k\la \la t_1,...,t_r\ra \ra$ by $t^{[s_i]} \mapsto t^{[s_i]}$, which we denote by $f\mapsto \overline{f}$. Examining the coefficients in the above calculation, we find that within $\mathfrak{m}_L\ti...\ti \m_L$, the value of ${f(x)}$ modulo $\mathfrak{m}_L^k$ can be recovered from $\overline{f}$ as $\overline{f}(\overline{x})$. In particular, if the mod $\m_L^k$ period map $\overline{\Phi}$ is non-constant, then its reduction in terms of divided powers power series is non-constant.
	
Suppose that $\im \Phi$ is not contained in some proper closed algebraic subspace $Z\subset \mathcal{H}$. Upon restriction to an open affine $U\subset \mathcal{H}$ containing\footnote{Note $\Phi(\Omega)$ is contained in a single residue disk by compatibility with the mod $p$ period map.} $\Phi(\Omega)$, there exists some regular algebraic function $f$ vanishing on $Z\cap U$ and non-vanishing on $\im \Phi$. After trivialising $\Omega$, the composite $F=f \circ \Phi$ is then a non-zero element of $\OO_L\la \la t \ra \ra$.
		\begin{lemma}\label{l:newtonpoly}
		Let $F \in \OO_{L}\la \la t \ra \ra$ be non-zero. Then $F$ has finitely many zeroes in the residue disk $\mathfrak{m}_L$. Moreover, if $\bar{F} \not \equiv 0 \in \OO_{L}/\mathfrak{m}_L^k \la \la t \ra \ra$, then the number of zeroes within $\mathfrak{m}_L$ is bounded in terms of only $\overline{F}$.
	\end{lemma}
	\begin{proof}
		This is an easy Newton polygon argument. The number of roots of $F$ of valuation $e$ is equal to the length of the slope $-e$ segment of the Newton polygon of $F$. The possible valuation of the coefficients of elements of $\OO_L\la \la t \ra \ra$ is minimised at the $(p^r)$\textsuperscript{th}-terms. From \eqref{eq:valuationoffactorial}, we find that the Newton polygon of $F$ lies above the line through the origin of gradient $-1/(p-1)>-1$ except possibly at the origin (this uses that $p\neq 2$ and $L$ is unramified). In particular, the segments of slope $\le -1$ have finite length, i.e.\ there are only finitely many zeroes within $\mathfrak{m}_L$.
		
		Now assume that $F= \sum_{n\ge 0} a_n \frac{t^n}{n!}$ has $\overline{F}\not \equiv 0$. Either $v_p(a_0)=0$ or some $s>0$ has $v_p(a_s)\le k-1$. Combining either of these with the need to lie above the line of gradient $-1/(p-1)$ bounds the possible length of slopes $\le -1$ after the $s$\textsuperscript{th}-point in terms of only $\overline{F}$.
	\end{proof}
So we obtain finiteness of points with Galois representation lying in a fixed isomorphism class as soon as the corresponding $Z(\phi^{[L:\q_p]})$-orbit does not contain the image of $\Phi$. Exhibiting families with this property is the major part of \cite{LV}.

\section{Period mappings in families}
In view of Lemma \ref{l:newtonpoly}, we now work to make rigorous the following claim: If the families $\mathcal{Y}\to \mathcal{X}$ for varying choices of $\mathcal{X}$ can themselves be assembled into a smooth family over some base $\mathcal{M}$, then for sufficiently close fibres the period mapping can be identified modulo $p^k$ and the points of $\mathcal{X}_s$ for $s \in \mathcal{M}$ with a fixed Galois representation can be exhibited as lying within the zero set of functions $F_s$ with common reduction modulo $p^k$.

Now suppose we have a family $\mathcal{Y}\overset{\pi}{\to} \mathcal{X}\overset{u}{\to} \mathcal{M}$ of smooth $\OO_L$-varieties with $u$ is smooth and $\pi$ is smooth projective and with the fibres over $s \in \mathcal{M}(\OO_L)$ themselves being families in the sense of Section \ref{s:background}. In this situation, when given a residue disk $\Omega\subseteq \mathcal{X}(\OO_L)$, we denote by $\Omega_s$ the fibre over $s\in u(\Omega)$.

\begin{lemma}\label{l:abstract}
	Let $\mathcal{Y}\underset{\pi}{\overset{\mathrm{sm.\ proj.}}{\longrightarrow}} \mathcal{X} \underset{u}{\overset{\mathrm{sm.\ curve}}{\longrightarrow}} \mathcal{M}$ be a family. Fix a residue disk $\Omega\subseteq \mathcal{X}(\OO_L)$ and a base point $s\in u(\Omega)$. Fix a closed algebraic $\OO_L$-subspace $Z\subset \mathcal{H}$ and suppose that the period mapping $\Phi|_{\Omega_s}$ does not have image contained within $Z$. Then there is some neighbourhood $U$ of $s$ for which if $s'\in U$, then $\im \Phi|_{\Omega_{s'}}$ is not contained in $Z$ and $(\Phi|_{\Omega_{s'}})^{-1}(Z(\OO_L))$ is uniformly bounded independent of $s'$.
\end{lemma}
\begin{proof}
	Since $Z$ is closed, there is some $k>0$ for which the image of $\overline{\Omega}_s$ under the mod $\m_L^k$ period map is not contained in the reduction of $Z$. As in the previous section, let $U\subset \H$ be some open affine containing $\im \Phi$ and fix some $f\in \OO_{U}(U)$ vanishing on $Z\cap U$ for which $\overline{f}$ is non-vanishing on the image of $\overline{\Phi |_{\Omega_s}}$.
	
	Since $u$ is smooth, the map of residue disks $\Omega\to u(\Omega)$ can be trivialised to identify fibres and under this $\overline{\Phi|_{\Omega_{s'}}}$ will be equal modulo $\mathfrak{m}_L^k$ for nearby $s'$. More specifically, if $\Omega$ lies above $w \in \mathcal{X}(k_L)$, then there exists a $\hat \OO_{\M,u(w)}$-algebra isomorphism $\hat \OO_{\mathcal{X},w}\cong \hat \OO_{\mathcal{M},u(w)}[[z]]$ for some coordinate $z$ \cite[(17.5.3)]{EGAIV4}. By the same argument, $\hat \OO_{\M, u(w)}$ is itself isomorphic to $\OO_L[[t_1,...,t_r]]$, so after choosing coordinates for $U\subset \H$, we can consider $\Phi|_{\Omega}$ as defined by a sequence of divided powers power series $(G_1,...,G_m)$ for $G_i \in \OO_L\la \la t_1,...,t_r,z\ra \ra$. In this concrete setting, the restriction $G_i|_{\Omega _{s'}}\in \OO_L\la\la z \ra \ra$ to fibre above $s'=(s'_1,...,s'_r)$ is given by $t_i^{[a_i]}\mapsto s_i^{a_i}/(a_i)!$. Here we have used that $L$ is unramified and $p\neq 2$ to ensure convergence for all $s\in u(\Omega)$. These same conditions moreover imply that $\overline{G_i|_{\Omega _{s'}} }\equiv \overline{G_i|_{\Omega_s}}$ modulo $\mathfrak{m}_L^k$ if $\overline{s}\equiv\overline{ s'}$ modulo $\mathfrak{m}_L^k$.
	
	Since $\im \overline{\Phi |_{\Omega_s}}$ is not contained in $\overline{ Z}\cap U_{\OO_L/\mathfrak{m}_L^k}$, we find that $\im \Phi|_{\Omega_{s'}}$ is not contained in $Z$. Moreover, for $s' \in V$,
	\[ F \colon \mathfrak{m}_L\overset{\Phi|_{\Omega_{s'}}=(G_i|_{\Omega_{s'}})_i}{\longrightarrow} \mathcal{H}(\OO_L) \overset{f}{\longrightarrow} \OO_L \]
	is a divided powers power series with zero set containing $(\Phi|_{\Omega_{s'}})^{-1}(Z(\OO_L))$ for which $\overline {F}$ is independent of $s'$ and non-vanishing modulo $\mathfrak{m}_L^k$. Applying Lemma \ref{l:newtonpoly}, we then obtain our desired uniform bound.
\end{proof}
We shall later use this lemma in the slightly more general setting of period maps corresponding to summands of those described above.
\section{Hurwitz families}
The existence of a smooth family whose fibres recover the Kodaira--Parshin families constructed by Lawrence--Venkatesh is due to Bertin--Romagny. The following brief description also serves as a sketch of the construction of Kodaira--Parshin families.
\begin{theorem}[Bertin--Romagny \cite{BertinRomagny}] For any finite group $G$, the functor ${\mathscr{H}}\colon (\textrm{RedSch}/\z[1/|G|])^\opp\to \Set$ which assigns to a reduced scheme $S/\z[1/|G|]$ the category of pairs $(C'\to C, G\overset{\sim}{\to} \Aut_{C}(C'))$ where $C/S$ is a smooth proper genus curve of genus $g$ and $C'$ is a $G$-cover ramified at a single point has the structure of a (reduced) Deligne--Mumford stack. If $\mathscr{H}$ is non-empty, the forgetful map $\mathscr{H}\overset{\pi}{\to} \M_{g,1}$ assigning a $G$-cover $(C'\to C)$ to $C$ marked with its ramification divisor is proper, quasi-finite and étale. Moreover, if $G$ is centre-free, then $\pi \colon \mathscr{H}\to \M_{g,1}$ is representable, so that $\pi$ is finite étale.\end{theorem}
\begin{proof}In \cite[Def.\ 6.2.1]{BertinRomagny}, Bertin--Romagny introduce more general Hurwitz stacks, which are by construction Deligne--Mumford stacks. The link between these and functor above is explained in \cite[Pf.\ of Thm.\ 2.11]{Marco}. The morphism $\pi \colon \mathscr{H} \to \M_{g,1}$ is shown to be proper and quasi-finite in \cite[Prop.\ 6.5.2 ii)]{BertinRomagny} and is étale by \cite[Thm.\ 5.1.5]{BertinRomagny}. If $G$ is centre-free, since the automorphism group of a $G$-cover is isomorphic to $Z(G)$, the base change of $\mathscr{H}$ by a scheme $T\to \M_{g,1}$ will be a Deligne--Mumford stack fibred in setoids, i.e. an algebraic space. Since $\mathscr{H}_T$ is then proper and quasi-finite over a scheme, it is a finite étale cover of a scheme and so is itself a scheme. In other words, $\pi$ is representable. Note that a morphism of stacks is said to be finite if it is representable and finite in the sense of schemes (i.e.\ proper and quasi-finite).
\end{proof}
Over $\co$, the finite étale property of $\pi\colon \mathscr{H}\to \M_{g,1}$ corresponds to the fact that nearby curves of genus $g$ can be identified and noting that this identifies their covers also.

We shall henceforth assume that $G$ is centre-free. Moreover, we shall assume that $g$ is sufficiently large so that $G$-covers ramified at a single point exist and $\mathscr{H}$ is non-empty (for $\Aff(q)$ this holds for $g\ge 2$). The stack $\mathscr{H}$ comes equipped with a universal genus $g$ curve and ramified $G$-cover. We can summarise this in the sequence
\begin{equation}  C'_{\univ}\overset{\textrm{$G$-cover}}{ \longrightarrow} C_\univ \overset{\textrm{curve}}{ \longrightarrow} \mathscr{H} \overset{\textrm{finite \'etale}}{ \longrightarrow} \M_{g,1} \overset{\textrm{curve}}{ \longrightarrow} \M_g    \label{eq:relKodPar}\end{equation}
(the final map being the universal curve over $\M_g$). Note that, by construction, $C_\univ$ is the pullback of the universal genus $g$-curve $\mathcal{M}_{g,1}\to \M_g$ under the moduli map $\mathscr{H}\to \M_{g,1}\to \M_g$. For any smooth projective genus $g$ curve $Y$, the fibre of \eqref{eq:relKodPar} over the point $[Y]\in \M_g$ recovers the Kodaira--Parshin family given in \cite[Sec.\ 7]{LV}:
\[{ Z\overset{\textrm{$G$-cover}}{ \longrightarrow} Y'\ti Y \overset{\textrm{curve}}{ \longrightarrow} Y' \overset{\textrm{finite \'etale}}{ \longrightarrow}Y} \overset{\textrm{curve}}{ \longrightarrow} [Y] \]
(where the notation $Z\to Y' \to Y$ matches the notation of \cite[Def.\ 7.2]{LV}). As such, \eqref{eq:relKodPar} can be thought of as a relative Kodaira--Parshin family over the base $\M_g$. For our purposes, we require the base to be a fine moduli space, so base change \eqref{eq:relKodPar} by $\M_{g,[3]}\to \M_g$ where $\M_{g,[3]}$ is the moduli scheme over $\z[1/3|G|]$ corresponding to genus $g$ curves with marked full $3$-torsion on their Jacobian. (Note that since $\pi : \mathscr{H}\to \M_{g,1}$ is representable, the base changes of $\mathscr{H}, C_\univ,C'_\univ$ will then be schemes.)

If we specialise to the centre-free group $\Aff(q)$ for some prime $q$, then as in \cite[Sec.\ 7.2]{LV}, we may consider the reduced relative Prym of $C'_\univ \to C_\univ$ as curves over $\mathscr{H}$ and we obtain an abelian-by-finite family over $\M_{g,[3]}$:
\[ \mathcal{A} \overset{\textrm{ab.\ var.}}{\underset{\varpi}{\longrightarrow} }\mathscr{H}_{[3]} \overset{\textrm{finite \'etale}}{ \underset{\pi}{\longrightarrow} } \M_{g,1,[3]} \overset{\textrm{curve}}{ \underset{u}{\longrightarrow} } \M_{g,[3]} . \]
In future sections we shall denote the universal curve $ \M_{g,1,[3]}$ as $\mathcal{X}$ and the moduli space of ramified covers $\mathscr{H}_{[3]}$ by $\mathcal{Y}$. Note, the family is defined over $\z[1/(3q(q-1))]$.

\section{Reduction steps}

We now begin the proof of Theorem \ref{t:main} by reducing to a local statement.

\begin{remark}\label{r:size}For a finite extension $L/\q_p$ with $p\nmid 3q(q-1)$, Lawrence--Venkatesh \cite[Sec.\ 5]{LV} introduce the notion of the ``size'' of points $x\in \calX(\OO_L)$ as a measure of the orbit-size of the fibres $\calY_x$ of the Kodaira--Parshin family of parameter $q$ under the action of Frobenius. Since $\calY\to \calX$ is étale over $\OO_L$, the fibre $\calY_{x,L}$ decomposes as a disjoint union of unramified field extensions $\coprod_i \Spec L_i$. We then set
	\[ \textrm{size}(\calY_{x,L}) = \frac{\displaystyle \sum_{[L_i : L ]<8 } [L_i:L]  }{\displaystyle\sum_i [L_i : L] } . \]
	For our purposes, we shall simply say $x$ has \emph{low size} when some $L_i/L$ has degree at least $8$. This will ensure good variation of the local Galois representation of $\calA_{x,L_i}$.
	
	Lawrence--Venkatesh show that given a number field $F$ and finite set of places $S$, it is always possible to choose $q\ge 3$ such that there is some finite place $v\notin S\cup \{w \mid 3q(q-1) \}$ such that all $\OO_{F,S}$-points of a fixed smooth projective curve $\mathcal{C}/{\OO_{F,S}}$ of genus $g$ have low size \cite[Pf.\ of Thm 5.4]{LV} with respect to $\Frob_v$. Moreover, $v$ can be chosen to satisfy the technical condition of being \emph{friendly}\footnote{For reference: $v$ is said to be friendly if it is unramified and when $F$ contains a CM subfield, the prime that $v$ lies over within the totally real subfield $M^+$ of the maximal CM subfield $M$ of $F$ is inert in $M/M^+$.}. This condition was introduced by Lawrence--Venkatesh \cite[Def.\ 2.7]{LV} in order to remove the assumption that the Galois representations $H^1_\et(\mathcal{Y}_{x, \bar F},\q_p)$ are semisimple and can be ignored if one is willing to apply the semisimplicity result of \cite[Satz 3]{FaltingsSemisimplicity}. In any case, for our purposes friendly can be taken to be a black box. We do however, require slightly more about the existence of $q,v$, which can easily obtained by examining Lawrence--Venkatesh's proof:	
\end{remark}

\begin{lemma}\label{l:size}
	Fix $g\ge 2$, a number field $F$ and finite set of places $S$. Then there exists $q\ge 3$ and $v\notin S\cup \{w \mid 3q(q-1) \}$ such that for any smooth projective curve $\mathcal{C}/\OO_{F,S}$, all points of $\mathcal{C}(\OO_{F,S})$ have low size with respect to $\Frob_v$ and the Kodaira--Parshin family with parameter $q$. Moreover, $v$ can be chosen to be friendly.
\end{lemma}
\begin{proof}
In the proof of \cite[Thm 5.4]{LV}, they prove that $x \in \mathcal{C}(\OO_{F,S})$ has low size if $q$ is chosen to satisfy three conditions (i)-(iii), and $v$ is chosen to simultaneously satisfy three other conditions (i')-(iii'). Each of these conditions depends on $g,F,S$ but not the choice of $\mathcal{C}$ or $x \in \mathcal{C}(\OO_{F,S})$. Since, pairs $q,v$ satisfying these conditions are shown to exist, such choices will ensure low size for all points of all $\mathcal{C}$. Condition (i') is precisely that $v$ be friendly.
\end{proof}
 As a result, we can use size to impose a condition on local fields $L$: we say that $L$ satisfies the \emph{size condition of Lawrence--Venkatesh} with respect the Kodaira--Parshin family of parameter $q$ if all points of $\mathcal{X}(\OO_L)$ have low size with respect to $\Frob_L$.

\begin{prop}\label{p:reduct}
	Fix a choice of integer $q>0$ and unramified local field $L/\q_p$ with $p\nmid 3q(q-1)$ and for which $L$ satisfies the size condition of Lawrence--Venkatesh with respect to the Kodaira--Parshin family of parameter $q$. Then for any choice of residue disk $\Omega$ of $\X(\OO_{L})$, there exists a constant $N'=N'(L,g,\Omega)$ such that for any $s \in \pi(\Omega)$ and $\psi\colon G_L\to \GL_{2g}(L)$,
	\[\{ x \in  \Omega_s \subseteq\calX_s(\OO_{L}) \mid \rho_x \cong \psi \}\le N'. \]
	
\end{prop}
In other words, the fibres of the local period mapping on each $\mathcal{X}_s$ are bounded and this bound is independent of $\mathcal{X}_s$. In addition to this local statement, we require uniform control of the number of points $x$ for which $\rho_x$ is not semisimple. This can either come from the proof by Faltings that in fact all $\rho_x$ are semisimple \cite[Satz 3]{FaltingsSemisimplicity} or from further analysis of the period map:
\begin{prop}\label{p:uniformss}
	Let $F$ be a number field and $S$ be a finite set of primes. Then there exists a constant $N''=N''(F,S,g)$ such that
	\[\{ x \in  \mathcal{C}(\OO_{F,S})  \mid \rho_x \textrm{ is not semisimple}\}\le N'' \]
	for all smooth projective curves $\mathcal{C}/\OO_{F,S}$.
\end{prop}

\begin{proof}[Proof of Thm.\ \ref{t:main} assuming Prop.\ \ref{p:reduct}, \ref{p:uniformss}] Let $\mathcal{C}/\OO_{F,S}$ be a smooth projective curve of genus $g$. Then $\Jac(\mathcal{C})$ is an abelian variety/$\OO_{F,S}$. In particular, it has good reduction outside $S$, so $\Jac(\mathcal{C})$ attains full $3$-torsion over an extension of $F$ of degree at most $|\GSp_{2g}(\F_3)|$ unramified outside of $\{\p \mid 3 \} \cup S$. There are finitely many such fields (Hermite's Theorem) and so their compositum $K$ is a number field. If $S_K$ denotes the primes of $K$ lying over $S$, then $\mathcal{C}$, non-uniquely, defines an $\OO_{K,S_K}$-point $s\in \M_{g,[3]}(\OO_{K,S_K})$. Since $K$ is independent of $\mathcal{C}$, it suffices to bound $K$-rational points on curves with full $3$-torsion on their Jacobian and good reduction outside $S_K$.

	By Lemma \ref{l:size}, there is some $q\ge 3$ and choice of prime $v\notin S_K\cup \{w \mid 3q(q-1)\}$ for which $K_v$ satisfies the Lawrence--Venkatesh size condition with respect to the Kodaira--Parshin family of parameter $q$. Applying Proposition \ref{p:reduct} with $L=K_v$, we find that the number of points in a residue disk with fixed Galois representation is bounded.
	
	Let $v$ lie over $p$. Then for all $x\in \calX(\OO_{K,S_K})$, the representation  $H^1_\et(\calA_{x,\bar K},\q_p)=$\linebreak[4] $\bigoplus_{\Spec F_i \in \calY_x} \Ind^{G_K}_{G_{F_i}} H^1_\et (\calA_{x, \bar F_i},\q_p) $ is unramified outside $S_K\cup \{ w | p\}\cup \{ w | q(q-1)\}$. At this stage we may either apply \cite[Satz 3]{FaltingsSemisimplicity} to deduce all the $\rho_x$ are semisimple, or, for a more Lawrence--Venkatesh style argument, apply Prop.\ \ref{p:uniformss} to deduce the number of points $x$ for which $\rho_x$ is not semisimple is finite and uniformly bounded independent of $\mathcal{C}$. In any case, the Faltings--Serre method \cite[V.2]{FaltingsWustholz} then shows that, as $x$ runs over $\calX(\OO_{K,S_K})$, there are only finitely many isomorphism classes of $G_K$-representations taken by $\rho_x$. As a result, it suffices to separately bound the number of points for which $\rho_x$ restricts to a fixed isomorphism class of $G_{L}$-representations.

	Since the number of points on smooth projective curves over the residue field $k_L$ is bounded depending only on $g$, the number of residue disks of $\mathcal{C}$ is bounded independently of $\mathcal{C}$ and we obtain the desired uniform bound $N$.
\end{proof}

\section{Local argument}
It remains to manoeuvre ourselves into a position to apply Lemma \ref{l:abstract}. In the notation of Proposition \ref{p:reduct}, fix a residue disc $\Omega\subseteq \mathcal{X}(\OO_L)$ and basepoint $x_0 \in \Omega$. Let $E_0$ denote the $\OO_L$-algebra of regular functions on $\mathcal{Y}_{x_0}$. As $\pi$ is étale, $\mathcal{Y}_{x_0}=\Spec E_0$ decomposes as $\coprod_{i=1}^k y_i = \coprod_{i=1}^k \Spec \OO_{F_i}$ where $F_i/L$ are finite unramified extensions. The fact that $\pi$ is étale also ensures $\mathcal{Y}_x$ is identified with $\mathcal{Y}_{x_0}$ for all $x\in \Omega$.

For all $x\in \Omega$, via the Gauss--Manin connection, $H^1_\dR(\mathcal{A}_{x}/\OO_L) $ is canonically identified with $H^1_\dR(\mathcal{A}_{x_0}/\OO_L)$ as $\OO_L$-modules. Since $\pi$ is étale, the ring of regular functions on $\mathcal{Y}_x$ is similarly canonically identified as $x$ varies in $\Omega$. Now consider $H^1_\dR(\mathcal{A}_{x}/\OO_L) $ as a $H^0_\dR(\mathcal{A}_x/\OO_L)$-module. Since the identification of regular functions on $\mathcal{Y}_x$ and $\YY_{x_0}$ can also be seen as the Gauss--Manin connection on $H^0_\dR(\mathcal{A}_x/\OO_L)$, the identification of $H^1_\dR(\mathcal{A}_{x}/\OO_L) \cong H^1_\dR(\mathcal{A}_{x_0}/\OO_L)$ is in fact an identification of $E_0$-modules in the obvious sense. Moreover, this respects the polarisation $\omega$ on $\calA$ obtained from the Prym construction, which is of degree invertible in $\OO_L$.

As in explained in Section \ref{s:background}, we can record the variation of the Hodge filtration of $H^1_\dR(\mathcal{A}_{x}/\OO_L)$ within $H^1_\dR(\mathcal{A}_{x_0}/\OO_L)$ to obtain a period map
\[ \Phi \colon \Omega \to \H(\OO_L),  \]
which by the previous paragraph can be considered to take values within the (compact) period domain $\H:=\Res^{E_0}_{\OO_L} \LGr(H^1_\dR(\mathcal{A}_{x_0}/\OO_L),\omega)$ whose $\OO_L$-points correspond to maximal isotropic subspaces of $H^1_\dR(\calA_{x_0}/\OO_L)$. Since the polarisation, $E_0$-action etc.\ exist modulo $\m_L^k$, exactly as in Section \ref{s:background}, the reduction modulo $\m_L^k$ of $\Phi$ coincides with the corresponding mod $\m_L^k$ period map and $\H$ is a smooth $\OO_L$-scheme. Of course, $\H$ splits as a product of period domains $\H_i$ corresponding to the decomposition of $E_0$. The projection $\Phi_i(x)$ of $\Phi(x)$ to $\H_i$ recovers the filtered $F$-isocrystal $(H^1_\crys(\mathcal{A}_{y_i}/F_i),\phi ,F^\bullet  ) $ associated to $H^1_\et(\mathcal{A}_{\bar F_i}, \q_p)$ as a $G_{F_i}$-representation by $p$-adic Hodge theory. Since induction of Galois representations corresponds to restriction of scalars of isocrystals, the filtered $F$-isocrystal corresponding to $\rho_x=\bigoplus_i \Ind^{G_{L}}_{G_{F_i}}H^1_\et(\mathcal{A}_{x,\bar F_i}, \q_p)$ is then the direct sum of these isocrystals described above considered with $L$-coefficients.

Now assume that $L$ satisfies the Lawrence--Venkatesh size condition for our family, so that without loss of generality $[F_1:L]\ge 8$. This will ensure that the Galois representations arising from the Kodaira--Parshin family varies as $x$ varies within a single curve even upon restriction to the summands parametrised by $\H_1$. Let $Z(\phi^{[L:\q_p]})\subseteq \Res_{\OO_L}^{\OO_{F_1}}\GL(H^1_\dR(\calA_{y_{1}}/\OO_{F_1}))$ be the $\OO_L$-subscheme defined by $\OO_{F_1}$-linear automorphisms which centralise the now linear map $\phi^{[L:\q_p]}$. For $h\in \H_1(\OO_L)$, we write $Z\cdot h$ for the intersection of $\H_1$ and the scheme theoretic image of
\[Z\overset{\cdot h}{\rightarrow}\Res_{\OO_L}^{\OO_{F_1}}\Gr(H^1_\dR(\calA_{y_{1}}/\OO_{F_1})). \]
 By definition, $Z\cdot h$ is a closed $\OO_L$-subscheme of $\H_1$ and has the property that any $h'\in \H_1(\OO_L)$ which defines a filtered $F_1$-isocrystal isomorphic to that defined by $h$ lies within $(Z\cdot h)(\OO_L)$. Note we are not assuming that $h$ and $h'$ need be related by automorphisms preserving the pairing.

\begin{theorem}[Lawrence--Venkatesh]\label{t:LVmain}	Fix a choice of integer $q>0$ and unramified local field $L/\q_p$ with $p\nmid 3q(q-1)$ and for which $L$ satisfies the size condition of Lawrence--Venkatesh with respect to the Kodaira--Parshin family of parameter $q$. Choose a basepoint $x_0\in \X(\OO_L)$ with residue disk $\Omega$ and let $s\in u(\Omega)$ be a choice of marked smooth projective curve of genus $g$. The projection $\Phi_1\colon \Omega_s \to \H_1(\OO_L)$ of the period mapping to a factor corresponding to $F_i/L$ of degree $\ge 8$ has Zariski dense image, whilst $\dim_{\q_p} ((Z\cdot h)_L)<\dim_{\q_p}((\H_1)_L)$ for all $h \in \H_1$.
\end{theorem}
\begin{proof}
	The Kodaira--Parshin family over $\mathcal{X}_s$ has full monodromy \cite[Thm.\ 8.1]{LV}, so the image of $\Phi|_{\Omega_s}$ is Zariski dense in ${\mathcal{H}}$ and this remains true after projection. The second statement is as in \cite[Pf.\ of Lem.\ 6.2]{LV}.
\end{proof}

The heart of the argument is now given by:

\begin{lemma}\label{l:uniformpreimagePhi1}\begin{enumerate}[i)]
		\item Let $h\in \H_1(\OO_L)$ and $s \in u(\Omega)$. Then the preimage of $(Z\cdot  h)(\OO_L)$ under the period map
		\[ \Phi_1|_{\Omega_s} \colon \Omega_s \to \H_1(\OO_L)  \]
		is uniformly bounded independent of $h$ and $s$.
		\item Let $s \in u(\Omega)$ and $Z'\subsetneq \mathcal{H}_1$ be any proper closed subvariety, then the preimage $( \Phi_1|_{\Omega_s} ) ^{-1}(Z'(\OO_L))$ is uniformly bounded independent of $s$.
	\end{enumerate}
\end{lemma}

\begin{proof}For the second statement, by Theorem \ref{t:LVmain}, $\im \Phi_{1}|_{\Omega_s}$ is not contained in any proper closed algebraic subspace of $\mathcal{H}_1$, so we may apply Lemma \ref{l:abstract} to obtain a uniform bound in a neighbourhood of $s$. Since $u(\Omega)$ is compact, this is sufficient to obtain a global uniform bound.
	
	The first statement is similar. Since $\mathcal{H}_1(\OO_L)$ and $u(\Omega)$ are compact, it suffices to find open neighbourhoods of $U_h$ of $h$ and $V_s$ of $s$ for which $|(\Phi_1|_{\Omega_{s'}})^{-1}(Z\cdot h')|$ is bounded as $h',s'$ vary within $ U_h, V_s$. Let $U_h$ denote the residue class of $h$ modulo $\mathfrak{m}_L^k$. Then for any $h'\in U_h$, we have that $\overline{Z\cdot h'}=\overline{Z\cdot h} $. If we now fix $\bar f\in\OO_{U_{\OO_L/\mathfrak{m}_L^r}}(U_{\OO_L/\mathfrak{m}_L^r})$, for any $h' \in U_h$, there exists a choice of lift $f_{h'}$ of $\overline{f}$ vanishing on $Z\cdot h'$ (which necessarily does not vanish on the image of $\Phi_1 |_{\Omega_s}$). Proceeding as in Lemma \ref{l:abstract}, we then obtain for $s'$ sufficiently close to $s$, a divided powers power series
		 \[ F \colon \mathfrak{m}_L\overset{\Phi_1|_{\Omega_{s'}}}{\longrightarrow} \mathcal{H}_1(\OO_L) \overset{f_{h'}}{\longrightarrow} \OO_L \]
		  with zero set containing $(\Phi_1|_{\Omega_{s'}})^{-1}(Z\cdot h')$ and $\overline {F}$ constant and non-vanishing modulo $\mathfrak{m}_L^k$. Applying Lemma \ref{l:newtonpoly}, we then obtain our desired uniform bound.\end{proof}

\begin{proof}[{Proof of Prop.\ \ref{p:reduct}}]
	In order for $\rho_x\cong \psi$, there must be a summand of $\psi$ of dimension $2d[F_1:L]$ induced from $G_{F_1}$. By the Krull--Schmidt theorem, $\psi$ has finitely many such summands $W$ up to isomorphism, if any at all. For each $W$, by Frobenius reciprocity, there are finitely many $G_{F_1}$-representations $M$ for which $\Ind^{G_L}_{G_{F_1}}M \cong W$ (bounded in terms of $[F_1:L]$). For each possible $M$, fix a single $h\in \H_1(\OO_L)$ whose corresponding filtered $F$-isocrystal is isomorphic to that of $M$, if such an $h$ exists.
	
	Applying Lemma \ref{l:uniformpreimagePhi1} {\it i)} to each choice of $h$, there are only finitely many $x\in \Omega_s$ for which $H^1_\et(\calA_{x,\bar F_1},\q_p)$ is isomorphic to $M$, bounded independently of $s,\psi$. In conclusion, we obtain a uniform bound on the number of $x \in \Omega_s$ for which $\rho_x\cong \psi$.
\end{proof}

\begin{proof}[Proof of Prop.\ \ref{p:uniformss}] Replace $F,S$ with $K$ and $S_K$ as in the pf.\ of Thm.\ \ref{t:main}. Let $q\ge 3$ and $v\notin S_K\cup \{w \mid 3q(q-1)\}$ for which $K_v$ satisfies the Lawrence--Venkatesh size condition with respect to the Kodaira--Parshin family of parameter $q$. For any $x \in \mathcal{X}(\OO_{K,S_K})$, the fibre $\pi^{-1}(x)_K= \coprod_{i=1}^k (y_i)_K = \coprod_{i=1}^k \Spec K_i$ decomposes as a union of field extensions unramified outside of $S_K$. As such, there are only finitely many possibilities for the fields $K_i$ which appear. For each of these, by Faltings--Serre, there are only finitely many possibilities for semisimple Galois representations appearing as the corresponding summands of $\rho_x$. 
	
	Fix a residue disk $\Omega \subset \mathcal{X}(\OO_{K_v})$. The period domain for $\rho_{x}|_{G_{K_{v}}}$ splits as a product indexed by pairs $(i,w)$, where $w$ is a place of $K_i$ above $v$. By exactly the same reasoning as in the proof of Prop.\ \ref{p:reduct}, for any $s \in u(\Omega)$, there can only be finitely many such $x\in \Omega_s$ for which there is a choice of $y_i$ and place $w$ dividing $v$ such that $[K_{i,w}:K_v]\ge 8$ with the property that summand of $\rho_x|_{G_{K_v}}$ corresponding to $y_i$ is isomorphic to the restriction of one of the finitely many semisimple Galois representations described above, and this bound is uniform in $s\in u(\Omega)$.
	
	So assume otherwise, i.e.\
	\begin{itemize}
		\item[(*)] Within $\rho_x \cong \bigoplus_i \rho_{y_i}$, for every $i$ such that there is a place $w$ of $K_{i,w}$ dividing $v$ with $[K_{i,w}:K_v]\ge 8$, the representation $\rho_{y_i}$ is not semisimple.
	\end{itemize} To deal with such cases, Lawrence--Venkatesh introduce subspaces $(\H_{i,w})_{K_v}^\textrm{bad}\subseteq (\H_{i,w})_{K_v}$ whose $K_v$-points parameterise maximal isotropic ${K_{i,w}}$-subspaces $F_h\subset H^1_{\dR}(\calA_{y_{0,i}}/{K_{i,w}})$ for which there exists a Frobenius stable $K_{i,w}$-subspace $W$ satisfying
\[ \dim_{K_{i,w}} (F_h\cap W) \ge \frac{1}{2}\dim_{K_{i,w}} (W).  \]

Lawrence--Venkatesh prove the following facts about them:
	\begin{proposition}[{\cite[Lem.\ 6.3]{LV}}]\label{p:badsubspaces}
		The $(\H_{i,w})_{K_v}^\textrm{bad}\subseteq (\H_{i,w})_{K_v}$ are closed algebraic subspaces. If $[K_{i,w}:K_v]\ge 8$, then the containment is strict.
	\end{proposition}
	\begin{proposition}[{\cite[Sublem.\ of Lem.\ 6.1]{LV}}]\label{p:badliesinbad}
	If $x$ satisfies (*) and $v$ is friendly, then there exists some choice of $(i,w)$ with $[K_{i,w}:K_v]\ge 8$ for which $\Phi_{i,w}(x)\in (\H_{i,w})_{K_v}^\textrm{bad}({K_v})$.
	\end{proposition}
Now set $\H_{i,w}^\textrm{bad}$ to be the closure of $(\H_{i,w})_{K_v}^\textrm{bad}$ within $\H_{i,w}$.
By Proposition \ref{p:badsubspaces}, we may apply Lemma \ref{l:uniformpreimagePhi1} {\it ii)} to obtain a uniform bound on the elements of $\Omega_s\cap \X_s(\OO_{K,S_K})$ whose period image lies in any $\H_{i,w}^\textrm{bad}$ with $[K_{i,w}:K_v]\ge 8$. By Lemma \ref{l:size}, we may additionally assume that $v$ was chosen to be friendly, so by Proposition \ref{p:badliesinbad}, we have uniformly bounded the $x$ with non-semisimple $\rho_x$.
\end{proof}

\begin{remark}\label{r:CHM}We conclude with a brief sketch of an alternative proof of Theorem \ref{t:main} using the stronger input of non-Zariski density of rational points on any $\calX \ti_{\M_{g,[3]}}Z$ for an arbitrary subvariety $Z\subseteq \M_{g,[3]}$. This approach is in the form of Caporaso--Harris--Mazur \cite{CHM}. Both methods rely on the existence of a relative Kodaira--Parshin family and Lawrence--Venkatesh's study of its fibres.
	
	By the same reduction steps as above, it suffices to prove uniform boundedness of points whose Galois representation associated via the Kodaira--Parshin family lies in a fixed isomorphism class. Under the same hypotheses and notation as used previously, for any $u \in \MM_{g,[3]}(\OO_L)$, the Kodaira--Parshin family over $\X_u$ in fact satisfies the inequality \cite[Proof of Lemma 6.2]{LV}:
	\begin{equation} \dim\left( \overline{\im \Phi_p|_{\mathcal{X}_u}}^\Zar\right) > (3g-3+1) +\dim Z(\phi^{[L:\q_p]})
	\label{eq:LVCHM} \end{equation}
	($3g-3+1$ being the dimension of $\mathcal{X}$). The inequality certainly then holds with $\dim (\overline{\im \Phi_p|_{\mathcal{X}_u}}^\Zar)$ replaced by $\dim( \overline{\im \Phi_p}^\Zar)$. So the preimage within $\calX(\OO_L)$ of any $Z(\phi^{[L:\q_p]})$-orbit has ``expected'' dimension zero. By p-adic Bakker--Tsimerman \cite[Sec.\ 9.2]{LV}, if the preimage has exceptional dimension, it is contained within a proper closed subvariety $Z_1$. Let $\tilde Z_1$ denote the set of points $u$ of $\MM_{g,[3]}$ for which $\X_u\subseteq Z_1$. Any $\mathcal{X}_u$ for $u$ lying outside of $\tilde Z_1$ necessarily has finitely many $\OO_L$-points with a fixed Galois representation since $\X_u \cap Z_1$ is finite. Moreover, this bound is uniform for $u\notin \tilde Z_1$ as the map $Z_1 \cap \pi^{-1}(U_1)\to U_1$ is a finite map.
	
	It remains to consider $\X_u$ for $u \in \tilde Z_1$. Since $\tilde Z_1$ has strictly smaller dimension than $\MM_{g,[3]}$ and the inequality \eqref{eq:LVCHM} still holds for $\X_{\tilde Z_1} \to \tilde Z_1$ as it holds fibrewise, we may iterate to obtain a global uniform bound. 
	
	Geometrically, this control on the fibrewise monodromy has shown that the preimage of $Z(\phi^{[L:\q_p]})$ is contained in a closed subvariety $Z\subsetneq \XX$ which necessarily does not contain any fibre of $\XX\overset{u}{\to} \MM_{g,[3]}$. If one was able to show that \eqref{eq:LVCHM} in fact holds with $\dim (\overline{\im \Phi_p|_{\mathcal{X}_u}}^\Zar)$ replaced by $\dim( \overline{\im \Phi_p|_{Z_1}}^\Zar)$, then one would be able to apply $p$-adic Bakker--Tsimerman directly to $Z_1$ to deduce that $\Phi_p^{-1} (Z(\phi^{[L:\q_p]}))$ has finitely many points lying outside of a proper closed algebraic subspace $Z_2\subsetneq Z_1$. If one could show this inequality holds for any $Z_i$ that may arise, then one would be able to deduce finiteness of $\XX(\OO_{F,S})$ itself. This requires existence of sufficient ``horizontal'' monodromy, which is a much deeper result than the fibrewise direction used above.
\end{remark}

\bibliographystyle{amsalphainitials}

\bibliography{MyCollection}

\vspace{0.7cm}

\noindent \emph{E-mail address:}	\texttt{lemos.pj@gmail.com}

\noindent\textsc{\noindent Department of Mathematics, King's College London, Strand, London,\\ WC2R 2LS, UK and Heilbronn Insitute for Mathematical Research, Bristol, UK}\\
\noindent \emph{E-mail address:}	\texttt{alex.torzewski@gmail.com}

\newpage

\end{document}